\documentclass{article}
\usepackage{enumerate}
\usepackage{amssymb}
\usepackage{amsmath}
\usepackage{latexsym}
\usepackage{amsthm}
\usepackage{xypic}
\newtheorem{teononum}{Theorem}

\newtheorem{theorem}{Theorem}[section]
\newtheorem{lemma}[theorem]{Lemma}

\newtheorem{corollary}[theorem]{Corollary}
\newtheorem{corol}{Corollary}

\theoremstyle{definition}
\newtheorem{definition}[theorem]{Definition}

\theoremstyle{remark}
\newtheorem{remark}[theorem]{Remark}
\newcommand\Qbar{\overline{\mathbb{Q}}}
\newcommand\IQ{\mathbb{Q}}
\numberwithin{equation}{section}

\makeatletter
\newcommand{\subjclass}[2][2010]{%
  \let\@oldtitle\@title%
  \gdef\@title{\@oldtitle\footnotetext{#1 \emph{Mathematics Subject Classification.} #2}}%
}
\newcommand{\keywords}[1]{%
  \let\@@oldtitle\@title%
  \gdef\@title{\@@oldtitle\footnotetext{\emph{Key words and phrases.} #1.}}%
}

\begin{document}

\title{On the Northcott property and other properties related to polynomial mappings}

\author{Sara Checcoli\protect\footnote{Address: Institute of Mathematics,
  University of Basel,
  Rheinsprung 21,
  CH-4051 Basel,
  Switzerland. E-mail address: $\mathtt{sara.checcoli@unibas.ch}$ }\hspace{1.5mm}
  and Martin Widmer\protect\footnote{Address: Department for Analysis and Computational Number Theory, 
  Graz University of Technology,
  Steyrergasse 30/II,
  8010 Graz, Austria. E-mail address: $\mathtt{widmer@math.tugraz.at}$}
 }

\date{}


\maketitle
\begin{abstract}
We prove that if $K/\IQ$ is a Galois extension of finite exponent and $K^{(d)}$ is the compositum of all extensions of $K$ of degree at most $d$, 
then $K^{(d)}$ has the Bogomolov property and the maximal abelian subextension of $K^{(d)}/\IQ$ has the Northcott property.

Moreover, we prove that given any sequence of finite solvable groups $\{G_m\}_m$ there  exists a sequence of Galois extensions $\{K_m\}_m$ with 
$\text{Gal}(K_m/\IQ)=G_m$ such that the compositum of the fields $K_m$ has the Northcott property. 
In particular we provide examples of fields with the Northcott property with uniformly bounded local degrees but not contained in $\IQ^{(d)}$.

We also discuss some problems related to properties introduced by Liardet and Narkiewicz to study polynomial mappings. Using results on the 
Northcott property and a result by Dvornicich and Zannier we easily deduce answers to some open problems
proposed by Narkiewicz.
\end{abstract}

\section{Introduction}
A subset $A$ of the algebraic numbers $\Qbar$ is said to have the \emph{Northcott property} or \emph{property (N)} if the set 
$$A(T):=\{x\in A | h(x)<T\}$$
is finite for every positive real number $T$ where $h$ denotes the absolute logarithmic Weil height on $\Qbar$ as defined in \cite{BoGu}
or \cite{Lang}.

For a subfield $K$ of $\Qbar$ and a positive integer $d$ we write
$K^{(d)}$ for the composite field of all extensions of $K$ in $\Qbar$ of degree at most $d$ over $K$. Moreover, we write $K_{\text{ab}}$
for the maximal abelian subextension of $K/\IQ$.

By a well-known theorem of Northcott (see Theorem 1 in \cite{Northcott}) any set $A\subseteq \Qbar$ of elements of uniformly bounded degree over $\IQ$
(i.e., there exists a constant $D$
such that $[\IQ(a):\IQ]\leq D$ for all $a\in A$) has property (N). In particular any number field
has property (N). 
Property (N) for certain infinite extensions of $\IQ$ has first been studied by Bombieri and Zannier in \cite{Zan}. 
In particular
they proposed the following question:
\begin{enumerate}[]
\item \textbf{Question 1} (Bombieri-Zannier): 
Does the field $\IQ^{(d)}$ have property (N)?
\end{enumerate}
For $d>2$ this is an open problem while for $d=2$ the answer is yes.  More generally 
Bombieri and Zannier showed
that the maximal abelian subextension of $K^{(d)}/K$ has property (N) for arbitrary number fields $K$. In particular, for $K=\IQ$, we have the following theorem.
\begin{teononum}[Bombieri-Zannier, 2001]\label{BZ}
Property (N) holds for ${\IQ^{(d)}}_{\text{ab}}$, for any $d$.
\end{teononum}
Property (N) for infinite extensions of $\IQ$ has also been studied by the second author in \cite{Widmer} where he
proved a simple but fairly robust criterion for a field to have property (N). We will use this criterion in Section \ref{proof_main}.

Recall that an algebraic extension $K/\IQ$ is said to have \emph{uniformly bounded local degrees} if there exists a positive integer $b$ 
such that, for every prime number $p$ and every valuation $v$ of $K$ above $p$, the completion $K_v$ is an extension of $\IQ_p$ of degree at most $b$.

A fact that is crucial in Bombieri and Zannier's proof of Theorem \ref{BZ} is the uniform boundedness of the local degrees of
${\IQ^{(d)}}_{\text{ab}}$. This is a consequence of the uniform boundedness of the local degrees of 
$\IQ^{(d)}$ (see Proposition 1 in \cite{Zan}). We shall prove that this property remains valid for $K^{(d)}$ whenever $K$ is a Galois extension of 
$\IQ$ with uniformly
bounded local degrees.  
Therefore it seems natural to address the following question: 
\begin{enumerate}[]
\item \textbf{Question 2}: If $K/\IQ$ is a Galois extension with uniformly bounded local degrees, does ${K^{(d)}}_{\text{ab}}$ have property (N)?
\end{enumerate}
At this point the reader should be warned: our definition of ${K^{(d)}}_{\text{ab}}$ 
coincides with Bombieri and Zannier's $K^{(d)}_{\text{ab}}$ from \cite{Zan}
if $K=\IQ$, but not in general.

In the recent paper \cite{Che} the first author gave the following characterization of Galois extensions of uniformly bounded local degrees. 
\begin{teononum}[Checcoli, 2011]\label{gal-exp} An algebraic Galois extension of the rationals has uniformly bounded local degrees if and only if its 
Galois group 
has finite exponent. 
\end{teononum}
Here we prove the following theorem which, due to Theorem \ref{gal-exp}, gives in particular an affirmative answer to Question 2.
\begin{teononum}\label{general}
Let $K$ be an algebraic extension of $\IQ$. If $K/\IQ$ is Galois then $K^{(d)}/\IQ$ is also Galois.

If moreover $\text{Gal}(K/\IQ)$ has finite exponent $b$, then $\text{Gal}(K^{(d)}/\IQ)$ has exponent at most $d! b$ and 
${K^{(d)}}_{\text{ab}}$ has property (N).
\end{teononum} 
Next let us consider the following ``stronger version'' of Question 1.
\begin{enumerate}[]
\item \textbf{Question 3}: Suppose $K$ is a subfield of $\Qbar$ with property (N). Does this imply that $K^{(d)}$ has property (N)
for any positive integer $d$?
\end{enumerate}
For a subfield $K$ of $\Qbar$ and positive integers $d$ and $n$ put $K^{(d)^{1}}=K^{(d)}$ and $K^{(d)^{n+1}}=(K^{(d)^{n}})^{(d)}$.
A positive answer to Question 3 would of course establish property (N) for ${\IQ}^{(d)^{n}}$. Unfortunately we don't even know
if ${\IQ}^{(2)^{2}}=({\IQ}^{(2)})^{(2)}$ has property (N). 
However, from Theorem \ref{general}
we easily deduce the following corollary.
Recall that  $({\IQ}^{(d)^{n}})_{\text{ab}}$ is the maximal abelian subextension of ${\IQ}^{(d)^{n}}/\IQ$.
\begin{corol}\label{cor1}
The field $({\IQ}^{(d)^{n}})_{\text{ab}}$ has property (N) for any positive integers $d$ and $n$.
\end{corol}
A weaker property than property (N), and also being discussed in \cite{Zan}, is the \emph{Bogomolov property}. We recall that a set of algebraic numbers $A$ 
is said to have the Bogomolov property if
there exists a positive real number $T_0$ such that the subset of non-zero elements of $A$ of height smaller than $T_0$ consists of all roots of unity 
in $A$. Here is another consequence of Theorem \ref{general}.
\begin{corol}\label{cor2}
Suppose $K/\IQ$ is an algebraic Galois extension of finite exponent.   
Then ${K}^{(d)^{n}}$ has the Bogomolov property for any positive integers $d$ and $n$.
\end{corol}
Special cases of this result have already appeared, at least implicitly, in the literature.
For $K$ a number field and $n=1$ Corollary \ref{cor2} follows immediately from a result of Bombieri and Zannier (Theorem 2 in \cite{Zan}),
and if $d=1$ then it follows immediately from  Theorem \ref{gal-exp} combined with Bombieri and Zannier's result.\\

Let us get back to property (N). Although the uniform boundedness of the local degrees of $\IQ^{(d)}$ was crucial in the proof of Theorem \ref{BZ}
it is not a necessary condition for property (N), as was shown by the second author (see, e.g., Corollary 2 in \cite{Widmer}).
On the other hand it is an open question whether every field in $\Qbar$ with uniformly bounded local degrees has property (N).
Recently Zannier and the first author (Theorem 1.1 in \cite{Che_Zan}) proved the existence of subfields of $\Qbar$ with uniformly bounded local degrees
that are not contained in $\IQ^{(d)}$ for any $d$. Thus an affirmative answer to the aforementioned question would go beyond
a positive answer to Question 1. 
Now here is very modest question in this direction. 
\begin{enumerate}[]
\item \textbf{Question 4}:  Are there extensions of $\IQ$ with property (N) and with uniformly bounded local degrees but 
which are not contained in $\IQ^{(d)}$ 
for any positive integer $d$? 
\end{enumerate}
Another important ingredient in the proof of Theorem \ref{BZ} is that 
the finite subextensions of ${\IQ^{(d)}}_{\text{ab}}/\IQ$ can be generated by elements of uniformly bounded degree which is not true for $\IQ^{(d)}$, 
provided $d$ is large enough, as shown recently by the first author (Theorem 2 in \cite{Che}). 
Again this is not a necessary condition for property (N) (again see Corollary 2 in \cite{Widmer}). 
However, these examples are not contained in $\IQ^{(d)}$ for any $d$. So we propose a second modest question.  
\begin{enumerate}[]
\item \textbf{Question 5}: Are there subfields of $\IQ^{(d)}$ with property (N) whose subfields of finite degree (over $\IQ$) cannot 
be generated by elements of 
uniformly  bounded degree, for some positive integer $d$?
\end{enumerate}
Positive answers to the last two questions follow from the group-theoretical constructions made in \cite{Che_Zan} and the following theorem
which might have some independent interest. For subfields $K_1,\ldots, K_{m}$ of $\Qbar$ we write $K_1\cdots K_{m}$ for the compositum of $K_1,\ldots,K_{m}$ in $\Qbar$.
\begin{teononum}\label{main}
Let $\{G_m\}_m$ be a sequence of finite solvable groups. Then there exists a sequence $\{K_m\}_m$ of finite Galois extensions of $\IQ$ such that: 
\begin{itemize}
\item  $\text{Gal}(K_m/\IQ)=G_m$ for every $m$;
\item $K_1\cdots K_{m}\cap K_{m+1}=\IQ$ for every $m$; 
\item the compositum of the fields $K_m$ has the Northcott property.
\end{itemize}   
\end{teononum}
Finally, in Section \ref{Properties} we discuss some properties such as property (P) introduced by Liardet and Narkiewicz to study polynomial mappings. 
It is well-known that property (N) implies property
(P), and we will show that some other properties
also follow from property (N). We use results 
about property (N) and a result of
Dvornicich and Zannier to easily  answer some open problems
proposed by Narkiewicz in \cite{Narkiewicz} regarding these
properties.

\section{Proofs of Theorem \ref{general}, Corollary \ref{cor1} and Corollary \ref{cor2}}\label{d-op}
We start by proving Theorem \ref{general} and then we will deduce Corollary \ref{cor1} and Corollary \ref{cor2}.
For convenience let us recall the statement of the theorem. 
\begin{theorem}\label{d_ab}
Let $K$ be an algebraic extension of $\IQ$. If $K/\IQ$ is Galois then $K^{(d)}/\IQ$ is also Galois.

If moreover $\text{Gal}(K/\IQ)$ has finite exponent $b$, then $\text{Gal}(K^{(d)}/\IQ)$ has exponent bounded by $d! b$ and the field
${K^{(d)}}_{\text{ab}}$ has property (N).
\end{theorem}
\begin{proof}
We first prove that if $K/\IQ$ is Galois then $K^{(d)}/\IQ$ is also Galois.
In order to do this we take $\sigma\in \text{Gal}(K/\IQ)$. We want to prove that for every embedding $\bar\sigma:K^{(d)}\hookrightarrow \Qbar$ such 
that $\bar{\sigma}|_{K}=\sigma$ one has $\bar\sigma(K^{(d)})\subseteq K^{(d)}$.

Now every element $\beta\in K^{(d)}$ is contained in a compositum $L_1\cdots L_s$ where the $L_i$ are extensions of $K$ of degree at most $d$ (over $K$). 
Therefore $\bar\sigma(\beta)\in \bar\sigma(L_1)\cdots \bar\sigma(L_s)$. Now we notice that $L_i=K(\alpha_i)$ for some $\alpha_i$ with minimal polynomial 
$f_i(x)\in K[x]$ of degree at most $d$. We have $$\bar\sigma(f_i(x))=\sigma(f_i(x))=g_i(x)\in K[x]$$ since $\sigma\in \text{Gal}(K/\IQ)$. Moreover, 
$\deg(g_i(x))\leq d$ and one of the roots of $g_i(x)$ is $\bar\sigma(\alpha_i)$ which thus belongs to $K^{(d)}$. 
Hence $$\bar\sigma(\beta)\in \bar\sigma(L_1)\cdots \bar\sigma(L_s)\subseteq K^{(d)}.$$

We now prove the second part of the theorem. Assume that $K/\IQ$ is Galois with $\exp(\text{Gal}(K/\IQ))=b$.

First of all we want to prove that the extension ${K^{(d)}}/\IQ$ has Galois group with finite exponent.
We take $\sigma \in \text{Gal}({K^{(d)}}/\IQ)$; then $\sigma^{b}|_K$ is the identity since $\text{Gal}(K/\IQ)$ has exponent $b$. 
Therefore $\sigma^b \in  \text{Gal}({K^{(d)}}/K)$ which is a group of exponent bounded by $d!$. Therefore $\sigma^{b (d!)}=1$.

We finally prove that ${K^{(d)}}_{\text{ab}}$ has property (N).
It is a Galois extension of $\IQ$ with Galois group of exponent bounded by $d! b$. Thus by virtue of Theorem 1.2 of 
\cite{Che}, it has uniformly bounded local degrees over $\IQ$. Since the extension is also abelian, from Proposition 2.1 of \cite{Che_Zan} we 
conclude that ${K^{(d)}}_{\text{ab}}$ is a subfield of $\IQ^{(m)}$ for some $m$ and therefore of ${\IQ^{(m)}}_{\text{ab}}$. But according to 
Theorem \ref{BZ}, the latter field has property (N) and so ${K^{(d)}}_{\text{ab}}$ has also property (N).
\end{proof}
Next let us recall and prove Corollary \ref{cor1}.
\begin{corollary}
The field $({\IQ}^{(d)^{n}})_{\text{ab}}$ has property (N) for any positive integers $d$ and $n$.
\end{corollary}
\begin{proof}
We notice that 
\begin{equation}\label{incl}
({\IQ}^{(d)^{n}})_{\text{ab}}\subseteq   \left(\left( \ldots \left(\left(\left(\IQ^{(d)}\right)_{\text{ab}}\right)^{(d)}\right)_{\text{ab}}\ldots 
\right)^{(d)}\right)_{\text{ab}}
\end{equation}
since every finite abelian extension $L$ in ${\IQ}^{(d)^{n}}$ consists of  a finite tower $\IQ\subseteq L_1\subseteq \ldots \subseteq L_n=L$ in which $L_i\subseteq ({L_{i-1}}^{(d)})_{\text{ab}}$ for every $i$.
Applying Theorem \ref{d_ab} recursively starting from $K=\IQ$ we obtain that the right-hand side field in (\ref{incl}) has property (N) and this concludes the proof.
\end{proof}
Finally, let us restate and prove Corollary \ref{cor2}.
\begin{corollary}
Suppose $K/\IQ$ is an algebraic Galois extension of finite exponent.   
Then ${K}^{(d)^{n}}$ has the Bogomolov property for any positive integers $d$ and $n$.
\end{corollary}
\begin{proof}
Theorem 2 in \cite{Zan} implies that if $L/\IQ$ is a Galois extension and there is a non-empty set of rational primes at 
which the local degrees of $L$ are finite, then $L$ has the Bogomolov property.

In view of Theorem 1.2 of \cite{Che} we conclude that if $L/\IQ$ has Galois group of finite exponent, then $L$ has the Bogomolov property. 
Now the corollary follows immediately from Theorem \ref{d_ab}.
\end{proof}

\section{Proof of Theorem \ref{main}}\label{proof_main}
We fix some notation.
Let $K\subseteq F\subseteq L$ be a tower of number fields. We denote by $\Delta_K$ the absolute discriminant of $K$, by $\Delta_{F/K}$ the 
(relative) discriminant and by $N_{F/K}$ the norm  map of the extension $F/K$. We recall that
$$\Delta_{L/K}=N_{F/K}(\Delta_{L/F}){\Delta_{F/K}}^{[L:F]}.$$
We will apply the following criterion, proved by the second author in \cite{Widmer}.
\begin{theorem}[Widmer, 2011]\label{theo_Widmer}
Let $K_0 \subsetneq K_1\subsetneq  K_2 \subsetneq \ldots$ be a nested sequence of number fields and set 
$L =\bigcup_{i=1} K_i$. Suppose that 
$$\inf_{K_{i-1}\subsetneq M \subseteq K_i} (N_{K_{i-1}/\IQ}(D_{M/K_{i-1}}))^{\frac{1}{[M:K_0][M:K_{i-1}]}}\longrightarrow \infty$$
as $i$ tends to infinity where the infimum is taken over all intermediate fields
$M$ strictly larger than $K_{i-1}$. 
Then the field $L$ has the Northcott property.
\end{theorem}
Moreover, we need the following lemma.
\begin{lemma}\label{realize}
Let $G$ be a finite solvable group. Then there exists a sequence of Galois extensions $\{K_i\}_i$ of $\IQ$ such that:
\begin{itemize}
\item $\text{Gal}(K_i/\IQ)=G$ for every $i$; 
\item $K_1\cdots K_{i-1}\cap K_{i}=\IQ$ for every $i>1$.
\end{itemize}   
\end{lemma}
\begin{proof} 
We first note that for every positive integer $n$ the group $G^n$ is also solvable, and thus, by Shafarevich's Theorem, it is realizable. 
This means that there exists a Galois extension $F_{n}/\IQ$ with $\text{Gal}(F_{n}/\IQ)={G}^n$. Now $F_n$ has $n$ subfields with Galois group 
isomorphic to $G$ which are pairwise linearly disjoint over $\IQ$, in particular their pairwise intersection is $\IQ$. 
Since this can be done for every $n$, we see that there exists an infinite family 
$\mathcal{F}=\{K_i\}_i$ of realizations of $G$ with $K_i\cap K_j=\IQ$ for $i\neq j$.

Now if we fix any finite extension $K/\IQ$ then it can only intersect with finitely many elements of $\mathcal{F}$, otherwise it would have 
infinite degree and this concludes the proof.
\end{proof}

Let $G$ be a finite solvable group. We denote by $\{K_i\}_i$ a sequence of realizations of $G$ as in Lemma \ref{realize}. 
Suppose $|G|>1$ and consider the following set of fields 
\begin{alignat*}1
\mathcal{F}:=\{H|\IQ\neq H\subseteq K_i \text{ for some } i\}. 
\end{alignat*}
By Hermite's Theorem there is only a finite number of number fields with absolute value of the discriminant below a given bound. Therefore the sequence 
$\{|\Delta_H|\}_{H\in \mathcal{F}}$ tends to infinity.

Now for every $H\in \mathcal{F}$ the power to which a prime number divides the discriminant $\Delta_H$ is bounded from above by  
$$2[H:\IQ]^2\leq 2|G|^2,$$ and so depends only on $|G|$ (see, e.g., Theorem B.2.12 in \cite{BoGu}). 
%(it is a very rough bound; the exact bound is $v_{p}(\Delta_K)\leq [H:\IQ]\left(1-\frac{1}{e}+v_{p}(e)\right)$ 
%where $e$ is the ramification index of $p$ in $H/\IQ$ and we have equality if $p$ is tamely ramified... well, H is not necessarily Galois)
Writing $p_{H}$ for the biggest prime number dividing $\Delta_{H}$, this implies that $p_{H_i}\rightarrow \infty$ for every sequence 
$H_1, H_2, \ldots$ of distinct elements of $\mathcal{F}$.
In particular we have that for every integer $B>0$ the set $$\{H\in \mathcal{F}|p_H\leq B\}$$ is finite. Therefore for every $B>0$, there exists 
an integer $n$ such that for every $i>n$ the field $K_i$ satisfies $p_H>B$ for every $\IQ\neq H\subseteq K_i$.

We now consider a sequence of finite solvable groups $\{G_i\}_i$. We choose realizations $K_i/\IQ$ of  $G_i$ such that for every $i>1$, 
\begin{alignat}1\label{cond2}
p_H>{p_{K_{i-1}}}^{|G_{i}|^2} \text{ for every } \IQ\neq H\subseteq K_{i}.
\end{alignat}
In particular we have
\begin{alignat*}1
K_1\cdots K_{i-1}\cap K_{i}=\IQ.
\end{alignat*}
Moreover, we have a nested sequence of number fields
$$L_0=\IQ\subsetneq L_1=K_1\subsetneq L_2=K_1\cdot K_2\subsetneq \ldots \subsetneq L_i=K_1\cdots K_i \subsetneq \ldots$$
and we want to prove that the field $L=\bigcup_i L_i$ (which equals the compositum of the fields $K_i$) has the Northcott property. 

In order to prove this we use Theorem \ref{theo_Widmer}. Consider an intermediate extension $L_{i-1}\subsetneq M\subseteq L_i$. 
Then $M$ is given by a compositum of the form
$M=L_{i-1}\cdot H$ where $\IQ\neq H=M\cap K_{i}\subseteq K_i$ (as $K_i/\IQ$ is Galois we can apply Corollaire 1 from \cite{Bourbaki9}
chap. v, \S 10).
We consider the tower $\IQ\subseteq H \subseteq M$. Here we have $$\Delta_M=N_{H/\IQ}(\Delta_{M/H}){\Delta_{H}}^{[M:H]},$$
and therefore $p_H^{[M:H]}\mid \Delta_M$. But as $\IQ\subseteq L_{i-1} \subsetneq M$ we have also
$$\Delta_M=N_{L_{i-1}/\IQ}(\Delta_{M/L_{i-1}}){\Delta_{L_{i-1}}}^{[M:L_{i-1}]}.$$
Now note that $p_{L_{i-1}}=\max\{p_{K_{1}},\ldots, p_{K_{i-1}}\}=p_{K_{i-1}}$. Thus
(\ref{cond2}) implies $p_H>p_{L_{i-1}}$, and so we conclude that $p_H^{[M:H]}\mid N_{L_{i-1}/\IQ}(\Delta_{M/L_{i-1}})$.
Therefore, and by (\ref{cond2}), we get
$$\inf_{L_{i-1}\subsetneq M\subseteq L_i}{N_{L_{i-1}/\IQ}(\Delta_{M/L_{i-1}})}^{\frac{1}{[M:\IQ][M:L_{i-1}]}}\geq 
{p_H}^{\frac{[M:H]}{[M:\IQ][M:L_{i-1}]}}>{p_H}^{\frac{1}{{|G_i|}^2}}>p_{K_{i-1}}.$$
The latter clearly tends to infinity and thus, by Theorem \ref{theo_Widmer}, the field $L=\bigcup_i L_i$ has the Northcott property.

\section{Answer to Question 4}\label{ss2}
In this section we affirmatively answer Question 4, that is, we show the existence of a subfield of $\Qbar$ with 
uniformly bounded local degrees which is not contained in $\mathbb{Q}^{(d)}$ for any $d$
but has property (N). 

To this end we will apply Theorem \ref{main} and the following result, which was proved in 
\cite{Che_Zan} (see proof of Theorem 1.1), and, with a slightly different construction, also in \cite{Che} (see proof of Theorem 2.(iii)).
\begin{theorem}[Checcoli-Zannier, 2011]\label{teo_Che_2} Let $p$ and $q$ be two odd primes with $p$ dividing $q-1$. Then there exists a family 
of finite non-abelian $pq$-groups $\{G_m\}_m$ such that, given any realization $K_m/\IQ$ for $G_m$,  the compositum $K$ of the family 
$\{K_m\}_m$ has uniformly bounded local degrees but it is not contained in $\IQ^{(d)}$ for any positive integer $d$.
\end{theorem}
The groups $G_m$ are finite groups of exponent $pq$. Their construction is due to A. Lucchini and it is based on certain $p$-groups called \emph{extraspecial groups} 
(i.e., $p$-groups in which the center equals the commutator subgroup and they are both of order $p$), and their modules. For details on 
extraspecial groups we refer the reader to Ch.A, \S 19 and \S 20 in \cite{Doerk}.

In view of Theorem \ref{main} we can realize the groups $\{G_m\}_m$ of Theorem \ref{teo_Che_2} with a family of number fields $\{K_m\}_m$ in 
such a way that the compositum $K$ of the fields $K_m$ has the Northcott property. This provides a field with the
Northcott property with uniformly bounded local degrees but not contained in $\IQ^{(d)}$ for any integer $d$.

\section{Answer to Question 5}\label{ss1} 
The goal of this section is to positively answer Question 5.
We will show that for any $d\geq 27$
there exists a subfield $K$ of $\IQ^{(d)}$ with property (N) whose 
finite subextensions of $K/\IQ$ cannot be generated by elements of uniformly bounded degree.

As in the previous section we will apply Theorem \ref{main}. We also need the following result which was proved in Subsection 3.3 of \cite{Che}.
\begin{theorem}[Checcoli, 2010]\label{q5}
Let $p$ be an odd prime number. Let $G$ be an extraspecial group of order $p^3$ and exponent $p$ and let $\mathcal{L}=\{L_i\}_i$ be a family of 
Galois extensions such that for every index $i$ $\text{Gal}(L_i/\IQ)=G$ and 
$L_{i+1}$ is linearly disjoint over $\IQ$ from the compositum $L_1\cdots L_{i}$. Then there exists a family $\{K_m\}_m$ of finite 
Galois extensions of $\IQ$ such that:
\begin{itemize}
\item for every $m$ the group $\text{Gal}(K_m/\IQ)$ is an extraspecial $p$-group of order $p^{2m+1}$ and exponent $p$;
\item every $K_m$ is a subfield of the compositum of $m$ fields of $\mathcal{L}$. In particular the compositum $K$ of the family 
$\{K_m\}_m$ is a subfield of $\IQ^{(p^3)}$;
\item the extensions $K_m$ cannot be generated over $\IQ$ by elements of uniformly bounded degree.
\end{itemize}
\end{theorem}
Let $G$ and $L_i$ be as in Theorem \ref{q5}. 
As $L_{i+1}/\IQ$ is Galois we have $L_{i+1}$ and $L_1\cdots L_{i}$ are linearly disjoint over $\IQ$ if and only if 
$L_1\cdots L_{i}\cap L_{i+1}=\IQ$ (see, e.g., p.35 in \cite{FriedJarden}). Clearly the group $G$
is solvable. 

We apply Theorem \ref{main} with $G_m=G$ to get a family of fields $L_i$ as in Theorem \ref{q5} 
whose compositum has property (N).
Then the field $K$ of Theorem \ref{q5} also has property (N)
but the finite subextensions of $K/\IQ$ cannot be generated by elements of uniformly bounded degree.

\section{Some remarks on properties concerning polynomial mappings}\label{Properties}
In this section we discuss some properties for fields that were introduced by Liardet and
Narkiewicz to study polynomial mappings. We will also answer some open problems proposed by Narkiewicz.\\

Let $K$ be a field and $n$ a positive integer.
Following Narkiewicz \cite{Narkiewicz} we say a polynomial mapping $f:K^n\rightarrow K^n$ sending $(x_1,...,x_n)\rightarrow
(f_1(x_1,...,x_n),...,f_n(x_1,...,x_n))$
is called \emph{admissible} if none of the $f_i$ are linear
and their leading forms do not
have a nontrivial common zero in the algebraic closure of $K$ (here, as in \cite{Narkiewicz}, linear polynomials 
includes constants, otherwise no infinite field could have property (SP) from Definition \ref{P}; just take $f(x,y)=(x+y^2,1)$ and $X=K\times\{1\}$). 

The following properties were introduced by Narkiewicz in \cite{Narkiewicz}.
\begin{definition}[Narkiewicz]\label{P}
A field $K$ is said to have property (SP) if for every $n$ and for every admissible
polynomial mapping  $f:K^n\rightarrow K^n$ the conditions 
$X\subseteq K^n$, $f(X)=X$ imply the finiteness of $X$. If this implication holds in
the case $n=1$ then $K$ has property (P).
\end{definition} 
Clearly an algebraically closed field cannot have property (P) (just take
$X=K$). On the other hand any number field
has property (P), as shown in the early paper \cite{Narkiewicz62} of Narkiewicz.

Nowadays it is well-known that for subfields of $\Qbar$ property (N) implies property (P)
(see p.534 in \cite{Dvo_Zan} or  Theorem 3.1 in \cite{Dvo_Zan2}). And, as remarked in \cite{Dvo_Zan},
the same holds for property (SP).
Dvornicich and Zannier used this, combined with the results on property (N),
to answers some open problems raised by Narkiewicz in \cite{Narkiewicz}, e.g., 
``Give a constructive description of
fields with property (P) or property (SP)'' or ``does 
$\IQ^{(2)}$ have property
(P) or (SP)?''. In fact Narkiewicz conjectured already in 1963 (see \cite{Narkiewicz63}, and also \cite{Narkiewicz71} Problem 10 (i)) that 
$\IQ^{(d)}$ has property (P) for any $d$, 
and for $d>2$ this old conjecture is still open.  

In \cite{HK-N} it was shown that property (SP) is preserved under finite
extensions. 
Another interesting question proposed in \cite{Narkiewicz} is whether (P) is also preserved under finite 
extensions of fields and moreover, whether property (P) is equivalent to property (SP). 
Dvornicich and Zannier in \cite{Dvo_Zan} also answered both of these questions in the negative by giving an ingenious construction
in the cyclotomic closure of $\IQ$. 
\begin{theorem}[Dvornicich-Zannier, 2007]\label{DvornicichZannier}
Let $p$ be a prime such that $p-1$ has an odd prime factor $l$. Let $K'$ be the
field generated by the roots of unity of $p$-power order and let $K$ be the 
unique subfield of $K'$ such that $[K':K]=l$. Then $K$ has property (P), $K'$ does
not have property (P) and $K$ does not have property (N).
\end{theorem} 
Instead of considering $f(X)=X$ as in property (P) one might go one step further and ask whether two polynomials $f_1, f_2 \in K[x]$ 
acquiring the same value set on an infinite
subset $X$ of $K$ necessarily have the same degree. For $K=\IQ$ a negative answer to this question was given by Kubota 
in \cite{Kubota72}.
Still with $K=\IQ$, Kubota also showed that if $\deg f_1<\deg f_2$, $f_1(X)\subseteq f_2(X)$ and the restriction of $f_1$ to $X$ is injective
then $X$ must be finite. This motivates the following definition introduced by Narkiewicz in \cite{Narkiewicz} on p.95.
\begin{definition}[Narkiewicz]
A field $K$ is said to have property (K) if for every $n$ the following implication holds true.
Let $\Phi:K^n\rightarrow K^n$ be an admissible
polynomial mapping and let $\Psi:K^n\rightarrow K^n$ be another polynomial mapping. Denote by $d$ the 
minimum of the degrees of the polynomials defining $\Phi$ and by $D$ the maximum of the degrees of the 
polynomials defining $\Psi$. If $d>D$, X is a subset of $K^n$ satisfying $\Psi(X)\subseteq \Phi(X)$
and the restriction of $\Psi$ to $X$ is injective, then $X$ is finite. 
\end{definition} 
In \cite{Liardet71} Liardet has introduced another property related to property (P).
We denote by $\overline{K}$ an algebraic closure of the field $K$.
\begin{definition}[Liardet]\label{Liard_p_bar}
A field $K$ is said to have property ($\overline{\text{P}}$) if for any non-linear
polynomial $f\in \overline{K}[x]$ and every subset $X\subseteq \overline{K}$ consisting of elements of uniformly bounded degree
over $K$, the inclusion $X\subseteq f(X)$ implies
the finiteness of $X$.
\end{definition} 
\begin{remark}\label{LN}
In \cite{Narkiewicz} Narkiewicz gives a different definition of property ($\overline{\text{P}}$), where 
$f\in \overline{K}[x]$ and $X\subseteq f(X)$ are replaced by $f\in K[x]$ and $f(X)=X$. He claims (see Theorem 10.7 of \cite{Narkiewicz}) that 
it follows immediately from the definition that property ($\overline{\text{P}}$) is preserved under finite extensions.
This is obvious for Liardet's original definition but we have not been able to verify this for Narkiewicz's definition. 
Narkiewicz also asks whether, with this definition, property ($\overline{\text{P}}$) is equivalent to property (P) (see  Problem XIX of \cite{Narkiewicz}).
However, due to Theorem \ref{DvornicichZannier} the answer to this question is certainly no, provided ($\overline{\text{P}}$) is preserved under finite extensions.
\end{remark}
An analogue of property (P) for rational functions was introduced by Narkiewicz in \cite{Narkiewicz}.
\begin{definition}[Narkiewicz]
A field $K$ is said to have property (R) if the following implication holds true:
if  $f$ is a rational function in $K(x)$ and there exists an infinite subset $X$ of
$K$ with $f(X)=X$
then there exist $\alpha, \beta, \gamma, \delta \in K$ such that $f(x)=(\alpha
x+\beta)/(\gamma x+\delta)$.
\end{definition} 
In fact a similar version of property (R) with 
$f(X)=X$ replaced by $X\subseteq f(X)$ was already introduced earlier by Liardet in \cite{Liardet71}.

Next we prove a simple lemma 
which could also be quickly deduced from Lemma 9.2. b) in \cite{Narkiewicz}.

Let us recall some basic facts.
For every rational function $f(x)\in \Qbar(x)$ there exists an effectively computable
constant $c_f>0$ depending only on $f$ such that
$h(f(x))\geq \deg f h(x)-c_f$ for all $x\in \Qbar$ not poles of $f$ (see, e.g., Corollary 3.3 in \cite{Zannier}). This implies
that if $\deg f\geq 2$ then $h(f(x))\geq (3/2)h(x)$ for 
any $x\in \Qbar$ with $h(x)\geq 2c_f$ and which is not a pole of $f$. 
\begin{lemma}\label{propinvfinite}
Suppose $f(x)\in\Qbar(x)$ of degree at least $2$, and suppose $X\subseteq \Qbar$ has property (N) and  is such that
$X\subseteq f(X)$.
Then $h(\alpha)\leq 2c_f$ for any $\alpha\in X$, in particular, $X$ is finite.
\end{lemma}
\begin{proof}
First note that $X\subseteq f^{(n)}(X)$ for all $n$.
Let 
\begin{alignat*}1
X_0&=\{\alpha\in X|h(f^{(n)}(\alpha))\rightarrow \infty \text{ as }n \rightarrow \infty\},\\
X_1&=X\backslash X_0. 
\end{alignat*}
Each $\alpha$ in $X_1$ has height at most $2c_f$, and so it suffices to show that the set $X_0$ is empty. 
Clearly $f^{(n)}(X_0)\cap X\subseteq X_0$ and $f^{(n)}(X_1)\cap X \subseteq X_1$, and
thus $f^{(n)}(X_0)\cap X=X_0$ for all $n$. Not forgetting the Northcott property we see that the latter is 
a contradiction if $X_0$ is non-empty and $n$ is large enough.
\end{proof}
The next result collects all the  known implications regarding the field properties introduced.
\begin{theorem}\label{propertiesimplications}
For arbitrary fields we have the following implications:
\begin{alignat*}1
\text{(K)}&\Rightarrow \text{(SP)}\Rightarrow \text{(P)}, \\
\text{($\overline{\text{P}}$)}&\Rightarrow \text{(P)}, \\
\text{(R)}&\Rightarrow \text{(P)}.
\end{alignat*}
Moreover, for subfields of $\Qbar$ property (N) implies all of the properties
(K), (SP), ($\overline{\text{P}}$), (R) and (P).
\end{theorem} 
\begin{proof}
The implications (SP) $\Rightarrow$ (P), ($\overline{\text{P}}$) $\Rightarrow$ (P) and  (R) $\Rightarrow$ (P) are obvious.
Next we prove that property (K) implies property (SP). 
To this end we take $\Psi=id$ the identity so that $\Psi$ is injective on any subset of $K^n$ and $D=1$. 
Now let $f=\Phi:K^n\rightarrow K^n$ be any admissible polynomial mapping. Hence $d>D$. Suppose 
$X$ is a subset of $K^n$ satisfying $f(X)=X$, and therefore, $X=\Psi(X)\subseteq \Phi(X)=f(X)=X$. 
As $K$ has property (K) we conclude $X$ is finite and therefore $K$ has also property (SP). \\

Now suppose $K$ is a subfield of $\Qbar$ with property (N).\\

Property (K) for the field $K$ follows immediately from Theorem 2 and Corollary 1 in \cite{HK-N} (the reader should be warned: in \cite{HK-N} the authors write property (VP) for property (K)), by observing that, if $K$ has property (N) and
$h_n:\Qbar^n\rightarrow [0,\infty)$ is the usual absolute non-logarithmic Weil height, then of course
also $K^n$, with respect to $h_n$, has property (N). 

Next we will show that $K$ has property ($\overline{\text{P}}$).
To this end suppose $f\in \Qbar[x]$ is a polynomial of degree at least $2$, let
$X\subseteq \Qbar$ be a set of elements of uniformly 
bounded degree over $K$ and suppose $X\subseteq f(X)$. We shall prove that 
$X$ is finite. By Theorem 2.1 of \cite{Dvo_Zan2} the set $X$ has property (N). Hence we
can apply Lemma \ref{propinvfinite} and the claim follows.

Finally, we will show that $K$ has property (R). So let $f$ be in $K(x)$ and not of the form
$(\alpha x+\beta)/(\gamma x+\delta)$ for 
$\alpha, \beta, \gamma, \delta \in K$. This means $\deg f\geq 2$. Now suppose
$X\subseteq K$ with $f(X)=X$. By Lemma \ref{propinvfinite}
we conclude that $X$ is finite and this proves that $K$ has property (R). 
\end{proof}
In \cite{Narkiewicz} Narkiewicz addressed also the following questions: \\
Is property ($\overline{\text{P}}$) equivalent to property (P)? ({\bf Problem XIX}).\\
Is property (K) equivalent to property (P)? ({\bf Problem XX}).\\
Although not explicitly mentioned in Dvornicich and Zannier's work, Theorem \ref{DvornicichZannier},
combined with older results of Liardet, Halter-Koch and Narkiewicz,  answers both of these questions.
\begin{corollary}\label{corDvornicichZannier}
Let $K$ be as in Theorem \ref{DvornicichZannier}. Then $K$ has property (P) but
none of the properties (K), (SP),
($\overline{\text{P}}$).
\end{corollary}
\begin{proof}
By Theorem \ref{DvornicichZannier} the field $K$ has property (P)
and the finite extension $K'$ of $K$ does not have property (P).
By Theorem \ref{propertiesimplications} each of the properties (K), (SP) and ($\overline{\text{P}}$) implies property (P).
Moreover, the properties (K), (SP) and ($\overline{\text{P}}$) are preserved under finite
extensions; for property (K) and (SP) this was shown in \cite{HK-N} (see Theorem 1), and
for property ($\overline{\text{P}}$) see Remark \ref{LN}. Hence
$K$ cannot have any of the properties (K), (SP) or ($\overline{\text{P}}$). 
\end{proof}
Narkiewicz also addresses the problem
to give a constructive description of
fields with property (R), and more specifically: does $\IQ^{(2)}$ have property (R)? (\cite{Narkiewicz} {\bf Problem XVIII}).
We already know that $\IQ^{(2)}$ has property (N), and therefore, by Theorem \ref{propertiesimplications}, it
has all of the properties in question.
\begin{corollary}
The field $\IQ^{(2)}$ has the properties (N), (K), (SP), ($\overline{\text{P}}$), (R) and (P). 
\end{corollary}
Clearly each number field has properties (K), (SP) and ($\overline{\text{P}}$).
It is also known that these properties are preserved under finite extension.
For fields of algebraic numbers one could therefore ask whether these properties are even preserved
under taking the compositum of two fields. Regarding this question we have the following answer.
\begin{corollary}
For fields of algebraic numbers the properties (N), (K), (SP), ($\overline{\text{P}}$), (R) and (P) are in general not preserved under 
taking the compositum of two fields. 
\end{corollary}
\begin{proof}
Corollary 2 in \cite{Widmer} provides fields $L_1, L_2$ with property (N) whose compositum $L_1L_2$
contains infinitely many roots of unity (see the example right after Theorem 5). By Theorem \ref{propertiesimplications}
we see that $L_1$ and $L_2$ have all of  the properties in question.
Taking $f(x)=x^2$ we conclude moreover, that $L_1L_2$ has infinitely
many preperiodic points under the polynomial mapping $f$. By Proposition 3.1 from \cite{Dvo_Zan2} 
we conclude that $L_1L_2$ does not have property (P). Hence, again by Theorem \ref{propertiesimplications}, it has 
none of the properties (N), (K), (SP), ($\overline{\text{P}}$), (R) and (P).
\end{proof}

\section*{Acknowledgements}
We are indebted to Pierre Liardet who clarified various points regarding polynomial mappings, pointed out important additional references,
and sent us relevant literature.
This work was initiated when the first author visited the Department for Analysis and Computational Number Theory 
at Graz University of Technology. She whould like to thank the Department for the invitation and the financial support.  
The second author was supported by the Austrian Science Fund (FWF): M1222-N13.


\begin{thebibliography}{10}
%\bibitem{Amo_Nuccio}\label{Amo_Nuccio} F. Amoroso, F.A.E. Nuccio, \textit{Algebraic
%numbers of small Weil\rq{}s height in CM-fields:
%On a theorem of Schinzel}. Journal of Number Theory, 122 (2007) 247-260.
\bibitem{BoGu}\label{BoGu} E. Bombieri, W. Gubler, \textit{Heights in Diophantine Geometry}. Cambridge University Press, 2006.
\bibitem{Zan}\label{Zan} E. Bombieri, U. Zannier, 
\textit{A note on heights in certain infinite extensions of $\IQ$}.
Rend. Mat. Acc. Lincei, 12, 2001, 5-14.
\bibitem{Bourbaki9}\label{Bourbaki9} N. Bourbaki, \textit{Alg\'ebre, Chapitre 4-5}. Hermann, Paris, 1959.
\bibitem{Che_Zan}\label{Che_Zan} S. Checcoli, U. Zannier,
\textit{On fields of algebraic numbers with bounded local degrees}.
C. R. Acad. Sci. Paris 349 (2011), no. 1-2, 11-14.
\bibitem{Che}\label{Che} S. Checcoli,
\textit{Fields of algebraic numbers with bounded local degrees and their properties}.
To appear in Trans. Amer. Math. Soc.
\bibitem{Doerk}\label{Doerk} K. Doerk, T. Hawkes, \textit{Finite Soluble
Groups}. De Gruyter, Berlin, 1992.
\bibitem{Dvo_Zan}\label{Dvo_Zan}R. Dvornicich, U. Zannier, \textit{Cyclotomic
Diophantine problems}. Duke Math. J. 139
(2007), no. 3, 527-554. 
\bibitem{Dvo_Zan2}\label{Dvo_Zan2}R. Dvornicich, U. Zannier, \textit{On the
properties of Northcott and Narkiewicz for fields of algebraic numbers}. Functiones
et Approximatio 39 (2008), no. 1, 163-173. 
\bibitem{FriedJarden}\label{FriedJarden} M. D. Fried, M. Jarden, \textit{Field Arithmetic, Third Edition}. Springer-Verlag Berlin Heidelberg, 2008.
\bibitem{HK-N}\label{HK-N} F. Halter-Koch, W. Narkiewicz, \textit{Finiteness
properties of polynomial mappings}. Math. Nachr., 159, (1992), 7-18.
\bibitem{Kubota72}\label{Kubota72} K. K. Kubota, \textit{Note on a conjecture of W. Narkiewicz}. J. Number Theory, 4 (1972), 181-190.
\bibitem{Lang}\label{Lang} S. Lang, \textit{Fundamentals of Diophantine Geometry}. Springer, 1983.
%\bibitem{Liardet70}\label{Liardet70}P. Liardet, \textit{Transformationes rationelles
%et ensemble alg{\'e}briques}. Th{\'e}se 3 cycle, Marseille 1970.
\bibitem{Liardet71}\label{Liardet71}P. Liardet, \textit{Sur les transformationes
polynomiales et rationelles}. S{\'e}m. Th. Nombres Bordeaux, 1971/72, exp. 29.
\bibitem{Narkiewicz62}\label{Narkiewicz62} W. Narkiewicz, \textit{On polynomial transformations}.
Acta. Arith. 7 (1961/1962), 241-249.
\bibitem{Narkiewicz63}\label{Narkiewicz63}W. Narkiewicz, \textit{Problem 415}.
Colloq. Math. 10 (1963), no. 1, 186.
\bibitem{Narkiewicz71}\label{Narkiewicz71}W. Narkiewicz, \textit{Some unsolved problems}.
Colloque de Th\'eorie des Nombres (Univ. Bordeaux, Bordeaux, 1969), pp. 159-164.
Bull. Soc. Math. France, Mem. No. 25, Soc. Math. France, Paris, 1971.
\bibitem{Narkiewicz}\label{Narkiewicz} W. Narkiewicz, \textit{Polynomial Mappings}.
Lecture Notes in Mathematics 1600, Springer, 1995.
\bibitem{Northcott}\label{Northcott} D. G. Northcott, \textit{An inequality in the theory of arithmetic on algebraic varieties}. Proc.
Cambridge Phil. Soc. 45 (1949), 502-509 and 510-518.
\bibitem{Widmer}\label{Widmer} M. Widmer, \textit{On certain infinite extensions of
the rationals with
Northcott property}. Monatsh. Math., 162 (2011), no. 3, 341-353.
\bibitem{Zannier}\label{Zannier} U. Zannier, \textit{Lecture Notes on Diophantine Analysis }(with an Appendix by F. Amoroso). Appunti. Scuola Normale Superiore di Pisa (Nuova serie), Vol.8,
Edizioni Della Normale, Pisa, 2009.

\end{thebibliography}
\end{document}